\documentclass[a4paper]{article}
\usepackage{typearea}
\typearea{18}
\usepackage{amsmath,amssymb,amsthm}
\usepackage{tabularx}
\usepackage{graphicx}
\usepackage{bm}

\newtheorem{thm}{Theorem}[section]
\newtheorem{prop}{Proposition}[section]

  \makeatletter
    
    \@addtoreset{equation}{section}
  \makeatother
\date{}
\title{On a continued fraction expansion of the special function $e^{x}E_{1}(x)$ and an explicit expression of the continued fraction convergents}
\author{Naoki Murabayashi \and Hayato Yoshida}
\begin{document}
 \maketitle
\begin{abstract}
 In [1], there are given two representations of the function $G_{1}(x)$ (equal to the exponential integral $E_{1}(x)$) that appears in an expression of the first derivative of the $L$-function of an elliptic curves defined over $\mathbb{Q}$ at 1. One is the Puiseux Series, and the other is the continued fraction representation. In [3] we can see how to construct formally this continued fraction from $F(x):=e^{x}E_{1}(x)$ (explained briefly in Introduction) but we have never seen a proof of that it converges to the original function $F(x)$. More precisely an asymptotic expansion of $F(x)=e^{x}E_{1}(x)$ is $\displaystyle\sum_{k=1}^{\infty} (-1)^{k-1}(k-1)!\left( \displaystyle\frac{1}{x}\right)^{k}$ and this gives the continued fraction by quotient-difference algorithm, which is briefly announced by E. Stiefel and developed by H. Rutishauser.\\
 \ \ \ \ In this paper we define ``a continued fraction expansion of $F(x)$ at infinity'', which is analogous to the regular continued fraction expansion of real numbers, and prove that this expansion gives the same continued fraction. Moreover, we give concrete representations of rational functions which are obtained by truncating the continued fraction.
\end{abstract}
\section*{Keywords}
Continued fractions, Asymptotic expansion, Quatient-difference algorithm, Reccurence relation, Linear fractional transformations.
\section*{Declarations}
\textbf{Funding} The authors did not receive support from any organization for the submitted work.\\ \\
\textbf{Conflicts of interest/Competing interests} The authors have no conflicts of interest to declare that are relevant to the content of this article.\\ \\
\textbf{Availability of data and material} Not applicable.\\ \\
\textbf{Code availability} Not applicable.\\ \\
\textbf{Authors' contrinutions} All authors contributed to the study conception and design.\\ \\
\textbf{Ethics approval} Not applicable.\\ \\
\textbf{Consent to participate} Not applicable.\\ \\
\textbf{Consent for publication} Not applicable.
\newpage
\section{Introduction}
\ \ \ \ Let $\displaystyle\sum_{n=1}^{\infty} a_{n}n^{-s}$ be the Dirichlet series of the Hasse-Weil $L$ functions $L (E, s)$ of an elliptic curve $E$ over $\mathbb{Q}$. Putting $$\Lambda(s):=\left(\frac{\sqrt{N}}{2\pi}\right)^{s}\Gamma(s)L(E,s),$$ it has the following integral representation:$$\int_0^\infty f_{E}\left(\frac{iy}{\sqrt{N}}\right)y^{s-1}dy,$$
where $N$ is the conductor of $E$ and $f_{E}(\tau)=\displaystyle\sum_{n=1}^{\infty} a_{n}e^{2\pi in\tau} (\tau\in\mathbb{C}, \rm{Im}(\tau)>0)$ is a cusp form of weight 2 for $\Gamma_{0}(N)$ according to the Taniyama-Shimura conjecture, which has been proved by C. Breuil, B. Conrad, F. Diamond, R. Taylor, and A. Wiles. Therefore we obtain that
$$\Lambda(s)=\int_1^\infty f_{E}\left(\frac{iy}{\sqrt{N}}\right)(y^{s-1}-\varepsilon y^{1-s})dy,$$
where $\varepsilon \in \{\pm 1\}$ is determined by the action of the main involution $W_{N}$ on $f_{E}$, i.e., $f_{E}|_{[W_{N}]_{2}}=\varepsilon f_{E}$. The $r$ th derivative of $\Lambda(s)$ at $s=1$ is given by
$$\Lambda^{(r)}(1)=2\int_1^\infty f_{E}\left(\frac{iy}{\sqrt{N}}\right)(\log y)^{r}dy=2\sum_{n=1}^{\infty} a_{n}\int_1^\infty e^{-2\pi ny/\sqrt{N}}(\log y)^{r}dy.$$
Let $r$ be the order of vanishing of $L(E,s)$ at $s=1$. We have that $$\Lambda^{(r)}(1)=\frac{\sqrt{N}}{2\pi}L^{(r)}(E,1),$$
by substituting $s=1$ into the $r$ th derivative of $\Lambda(s)=\left(\frac{\sqrt{N}}{2\pi}\right)^{s}\Gamma(s)L(E,s)$.
Therefore we obtain that
\begin{eqnarray*}
L^{(r)}(E, 1)&=&\frac{2\pi}{\sqrt{N}}\Lambda^{(r)}(1)\\
&=&\frac{4\pi}{\sqrt{N}}\sum_{n=1}^{\infty} a_{n}\int_1^\infty e^{-2\pi ny/\sqrt{N}}(\log y)^{r}dy.
\end{eqnarray*}
Using an integration by parts in $\displaystyle\int_1^\infty e^{-2\pi ny/\sqrt{N}}(\log y)^{r}dy$, we have that
$$L^{(r)}(E, 1)=2r\sum_{n=1}^{\infty} \frac{a_{n}}{n}\int_1^\infty e^{-2\pi ny/\sqrt{N}} (\log y)^{r-1}\frac{dy}{y}.$$
\ \ \ \ The function $G_{r}(x)$ is defined by $$G_{r}(x):=\frac{1}{(r-1)!}\int_1^\infty e^{-xy}(\log y)^{r-1}\frac{dy}{y}\ \ (r \ge 1).$$
Then we have that
$$L^{(r)}(E, 1)=2r!\sum_{n=1}^{\infty} \frac{a_{n}}{n}G_{r}\left(\frac{2\pi n}{\sqrt{N}}\right).$$ 
In [1], two representation of the $G_{1}(x)=\displaystyle\int_1^\infty e^{-xy}\frac{dy}{y}=\int_{x}^{\infty} \displaystyle\frac{e^{-t}}{t}dt=E_{1}(x)$ is given: 
\begin{equation}
\log \frac{1}{x}-\gamma+\sum_{n=1}^{\infty} \frac{(-1)^{n-1}}{n\cdot n!}x^{n}\ \ (\gamma: \rm{Euler's\ constant});
\end{equation}
\begin{equation}
\cfrac{e^{-x}}{x+\cfrac{1}{1+\cfrac{1}{x+\cfrac{2}{1+\cfrac{2}{x+\cfrac{3}{1+\cdots}}}}}}
\end{equation}
for the purpose of calculating an approximation of $L^{(1)}(E_{0},1)$, where $E_{0}:y^{2}=4x^{3}-28x+25$. 
However, a proof of these representations is not explained at all. \\
\ \ \ \ Now we explain a proof of (1.1). We put
$$H(x):=\log \frac{1}{x}+\sum_{n=1}^{\infty} \frac{(-1)^{n-1}}{n\cdot n!}x^{n}.$$
The first derivative of $H(x)$ is given by
\begin{eqnarray*}
H^{'}(x)&=&-\frac{1}{x}+\sum_{n=1}^{\infty} \frac{(-1)^{n-1}}{n!}x^{n-1} \\
&=& -\frac{1}{x}-\frac{1}{x}\sum_{n=1}^{\infty} \frac{(-x)^{n}}{n!}\\
&=&  -\frac{1}{x}-\frac{1}{x}(e^{-x}-1) \\
&=& -\frac{e^{-x}}{x}.
\end{eqnarray*}
On the other hand the first derivative of $G_{1}(x)$ is given by
\begin{eqnarray*}
G^{'}_{1}(x)&=& \int_1^\infty e^{-xy}(-y)\frac{dy}{y} \\
&=& -\int_1^\infty e^{-xy}dy \\
&=& -\frac{e^{-x}}{x}.
\end{eqnarray*}
Therefore we have that 
$$G_{1}(x)=H(x)+C,$$
where $C$ is a constant. By calculating $G_{1}(1)-H(1)$, 
we obtain $C=-\gamma$, therefore (1.1). \\
\ \ \ \ Next we explain an outline of a proof of (1.2). 
Let $\{c_{n}\}_{n=0}^{\infty}$ be a sequence of real numbers. The following quantity is called the Hankel determinant: for $k\ge 2$,
$$H_{k}^{(n)}:=\left|\begin{array}{cccc}
      c_{n} & c_{n+1} & \cdots &  c_{n+k-1} \\
      c_{n+1} & c_{n+2} & \cdots & c_{n+k} \\
      \vdots & \vdots & \ddots & \vdots \\
      c_{n+k-1} & c_{n+k} & \cdots & c_{n+2k-2}
    \end{array}
  \right|;$$
$H_{0}^{(n)}:=1$; $H_{1}^{(n)}:=c_{n}$. Let $\{e^{(n)}_{k}\}_{n,k\ge 0}$ and $\{q^{(n)}_{k}\}_{n,k\ge 0}$ be two sequence in $\mathbb{R}$ determined by the following recurrence relations:
$$e_{k+1}^{(n)}=q_{k}^{(n+1)}-q_{k}^{(n)}+e_{k}^{(n+1)},\ \ \ \ q_{k+1}^{(n)}=\frac{e_{k+1}^{(n+1)}}{e_{k+1}^{(n)}}\cdot q_{k}^{(n+1)}\ \ \ \ (n,k\ge 0)$$
with initial conditions: $e_{0}^{(n)}=0$, $q_{0}^{(n)}=\displaystyle\frac{c_{n+1}}{c_{n}}$\ \ $(n\ge 0)$. Recursively constructing $e_{k}^{(n)}$ and $q_{k}^{(n)}$ from the sequence $\{c_{n}\}_{n=0}^{\infty}$ is called the quotient-difference algorithm. With respect to the connection of $\{e_{k}^{(n)},\ q_{k}^{(n)}\}$ and Hankel determinants, the following is well known:
\begin{thm}
For any integer $k\ge 1$, it holds that
$$e_{k}^{(n)}=\frac{H_{k+1}^{(n)}H_{k-1}^{(n+1)}}{H_{k}^{(n+1)}H_{k}^{(n)}},\ \ \ \ q_{k-1}^{(n)}=\frac{H_{k}^{(n+1)}H_{k-1}^{(n)}}{H_{k}^{(n)}H_{k-1}^{(n+1)}}.$$
\end{thm}
Let $f(x)$ be a function defined on an open interval in $\mathbb{R}$. A formal power series $\displaystyle\sum_{n=0}^{\infty} c_{n}x^{-n}$ with respect to $x^{-1}$ is called an asymptotic expansion of $f(x)$ when for any integer $n\ge 1$, there exist positive constants $M_{n}$ such that 
$$\left|f(x)-\displaystyle\sum_{m=0}^{n} c_{m}x^{-m}\right|\le M_{n}|x|^{-(n+1)}.$$  Putting $c_{n}=0$ for $n<0$, we extend $H_{k}^{(n)}$ for $n<0$ and $k\ge1$. We assume that if $n+k>0$, $H_{k}^{(n)}\neq 0$. Then it is well known that to every formal power series $\displaystyle\sum_{n=0}^{\infty} c_{n}x^{-n}$
there corresponds the continued fraction of the form
$$\cfrac{d_{0}}{1+\cfrac{d_{1}}{x+\cfrac{d_{2}}{1+\cfrac{d_{3}}{x+\cfrac{d_{4}}{1+\cfrac{d_{5}}{\cdots}}}}}},$$
with $d_{0}=c_{0}$, $d_{2k-1}=-\displaystyle\frac{H_{k}^{(1)}H_{k-1}^{(0)}}{H_{k}^{(0)}H_{k-1}^{(1)}}=-q_{k-1}^{(0)}$, $d_{2k}=-\displaystyle\frac{H_{k+1}^{(0)}H_{k-1}^{(1)}}{H_{k}^{(1)}H_{k}^{(0)}}=-e_{k}^{(0)}$ such that rational functions obtained by truncating this continued fraction are Pad\'e approximants of  $\displaystyle\sum_{n=0}^{\infty} c_{n}x^{-n}$.\\
\ \ \ \ The formal power series $$\displaystyle\sum_{k=1}^{\infty} (-1)^{k-1}(k-1)!\left( \displaystyle\frac{1}{x}\right)^{k}=\frac{1}{x}\displaystyle\sum_{k=0}^{\infty} (-1)^{k}k!\left( \displaystyle\frac{1}{x}\right)^{k}\eqno(1.3)$$ is an asymptotic expansion of $F(x)=e^{x}E_{1}(x)$. This is proved in Proposition 2.1. We set $c_{n}:=(-1)^{n}n!$. By the quotient-difference algorithm, we have that for $k\ge 1$
$$q_{k-1}^{(n)}=-(n+k),\ \ \ \ e_{k}^{(n)}=-k.$$
Therefore we obtain that
$$d_{2k-1}=-q_{k-1}^{(0)}=k,\ \ \ \ d_{2k}=-e_{k}^{(0)}=k.$$
 Then we obtain the following continued fraction which is equivalent to (1.2):
$$\frac{1}{x}\cdot \cfrac{d_{0}}{1+\cfrac{d_{1}}{x+\cfrac{d_{2}}{1+\cfrac{d_{3}}{x+\cfrac{d_{4}}{1+\cfrac{d_{5}}{x+\cdots}}}}}}=\cfrac{1}{x+\cfrac{1}{1+\cfrac{1}{x+\cfrac{2}{1+\cfrac{2}{x+\cfrac{3}{1+\cdots}}}}}}.$$
By making each numerator one, this continued fraction is equal to
$$\cfrac{1}{m_{1}+\cfrac{1}{m_{2}+\cfrac{1}{m_{3}+\cfrac{1}{m_{4}+\cdots}}}}$$
where
$$m_{n}=\begin{cases}
x\ \ \ \ \ \mbox{if}\ n:\mbox{odd}, \vspace{3mm}\\
\displaystyle\frac{2}{n}\ \ \ \ \mbox{if}\ n:\mbox{even}.
\end{cases}$$
This continued fraction is also given by Johann Georg von Soldner. In [3] it is proved that this continued fraction converges uniformly on any compact set of the complex plane without the subset of the negative real numbers, i.e., $\mathbb{C}-\mathbb{R}_{<0}$. Therefore this defines a holomorphic function $\widetilde{F}(x)$ on $\mathbb{C}-\mathbb{R}_{<0}$. In [3] it is stated without proof that $\widetilde{F}(x)=F(x)$. However, this is not trivial for us because the radius of convergence of (1.3) is zero so there is no relationship between $F(x)$ and (1.3) as a function. In the forthcoming paper, we will show $\widetilde{F}(x)=F(x)$ by following the proof of that the regular continued fraction expansion of a real number $\alpha$ converges to $\alpha$. \\
\ \ \ \ In this paper we define ``a continued fraction expansion of $F(x)$ at infinity'', which is analogous to the regular continued fraction expansion of real numbers, and prove that this expansion gives the same continued fraction. Moreover we give an explicit expression of rational functions obtained by truncating the continued fraction of $F(x)$, i.e., putting
$$\displaystyle\frac{P_{n}(x)}{Q_{n}(x)}:=\cfrac{1}{m_{1}+\cfrac{1}{m_{2}+\cfrac{1}{m_{3}+\cfrac{1}{\ddots \cfrac{\ddots}{m_{n-2}+\displaystyle\frac{1}{m_{n-1}+\displaystyle\frac{1}{m_{n}}}}}}}}, \eqno(1.4)$$
where $P_{n}(x),\ Q_{n}(x)\in \mathbb{Q}[x]$, we calculate $P_{n}(x),\ Q_{n}(x)$.
\section{A continued fraction expansion of $F(x)$ at infinity}
\ \ \ \ We put $F(x):=e^{x}E_{1}(x)=\displaystyle\int_1^\infty e^{-x(y-1)}\frac{dy}{y}$. Then, we have the following.
\begin{prop}
Let $\mathcal{O}$ be the Landau symbol. Then
$$F(x)-\left(\sum_{k=1}^{n} (-1)^{k-1}(k-1)!\left(\frac{1}{x}\right)^{k}\right)=\mathcal{O}\left(\frac{1}{x^{n+1}}\right)\ \ \ \ (x\to \infty).$$
\end{prop}
\begin{proof}
We set $t:=x(y-1)$. Then $F(x)=\displaystyle\int_0^\infty e^{-t}\frac{1}{t+x}dt$. Performing an integration by parts $n$-times, we have that $$F(x)=\sum_{k=1}^{n} (-1)^{k-1}(k-1)!\left(\frac{1}{x}\right)^{k}+(-1)^{n}n!\int_0^\infty e^{-t}\frac{1}{(t+x)^{n+1}}dt.$$
\ \ \ \ Since $t$ is non-negative, the following inequality holds:
$$0<\displaystyle\int_0^\infty e^{-t}\frac{1}{(t+x)^{n+1}}dt< \displaystyle\int_0^\infty e^{-t}x^{-(n+1)}dt=x^{-(n+1)}\displaystyle\int_0^\infty e^{-t}dt=\frac{1}{x^{n+1}}.$$
\ \ \ \ Therefore $$F(x)-\left(\sum_{k=1}^{n} (-1)^{k-1}(k-1)!\left(\frac{1}{x}\right)^{k}\right)=\mathcal{O}\left(\frac{1}{x^{n+1}}\right)\ \ \ \ (x\to \infty).$$
\end{proof}
\vspace{5mm}
Using the identity obtained by setting $n=1$ in this proposition, we obtain that
$$\frac{F(x)}{\displaystyle\frac{1}{x}}=1+\frac{1}{x}\cdot \frac{\mathcal{O}\left(\displaystyle\frac{1}{x^{2}}\right)}{\displaystyle\frac{1}{x^{2}}}\rightarrow 1\ \ \ \ (x\to \infty).$$
i.e, $F(x)\sim x^{-1}$ as $x\rightarrow \infty .$ Putting $F_{1}(x):=\displaystyle\frac{1}{F(x)}$, we have that
$$F_{1}(x)\sim x\ \ \ \ (x\to \infty).$$
Next we put $F_{2}(x):=\displaystyle\frac{1}{F_{1}(x)-x}.$ 
Using the identity obtained by setting $n=1$ (resp. $n=2$) in the denominator (resp. the numerator), we have that
$$F_{1}(x)-x=\displaystyle\frac{1-xF(x)}{F(x)}=\displaystyle\frac{1-x^{2} \mathcal{O}\left(\displaystyle\frac{1}{x^{3}}\right)}{1+x\mathcal{O}\left(\displaystyle\frac{1}{x^{2}}\right)} \rightarrow 1\ \ \ \ (x\to \infty).$$ 
Then, we get that
$$F_{2}(x)\sim 1\ \ \ \ (x\to \infty).$$
We inductively define
$$F_{m}(x):=\frac{1}{F_{m-1}(x)-(\mbox{main term of}\ F_{m-1}(x)\ \mbox{at\ infinity}).}$$
\ \ \ \ By calculating $F_{m}(x)$ $(m=3, 4, 5, 6)$, we conjecture that for any positive integer $m$,
$$ F_{2m-1}(x)\sim x,\ F_{2m}(x)\sim \frac{1}{m}\ \ \ \ (x\to \infty).$$
The first result of this paper is to prove the conjecture. Therefore we can state the following theorem.
\begin{thm}
For any positive integer $m$, we have that
$$F_{2m-1}(x)\sim x,\ F_{2m}(x)\sim \frac{1}{m}\ \ \ \ (x\to \infty).$$
\end{thm} 
\section{Preliminaries}
\ \ \ \ In this section we state properties of the matrix needed to prove Theorem 2.1. Let $k$ be an integer with $k\ge 2$. We define the matrix  $A^{(k-1)}$ by
$$A^{(k-1)}:=(a_{ij})\ \ \ \ 1\le i,j \le k,$$
where $a_{ij}=(-1)^{i+j-2}(i+j-2)!$. \\
\ \ \ \ Let $A^{(k-1)}_{i, j}$ be the matrix obtained by deleting the $i$ th row and the $j$ th column of $A^{(k-1)}$. 
\begin{thm}
Let $m$ be an integer with $1\le m\le k$. Then we have that
$$\mathrm{det}A^{(k-1)}_{k, m}=(-1)^{k+m}\frac{1}{((m-1)!)^{2}}\cdot \frac{1}{(k-m)!}\mathrm{det}A^{(k-1)}.$$
\end{thm}
\begin{proof}
If $\mathrm{det}A^{(k-1)} \neq 0$, the inverse of $A^{(k-1)}$ is given by$$(A^{(k-1)})^{-1}=\frac{1}{\mathrm{det}A^{(k-1)}}\tilde{A}^{(k-1)},$$
where $\tilde{A}^{(k-1)}$ is the adjugate matrix of $A^{(k-1)}$. The $(i,j)$ entry of the $\tilde{A}^{(k-1)}$ is $(-1)^{i+j}\mathrm{det}A^{(k-1)}_{j,i}$. In particular, the $k$ th column of $\displaystyle\frac{1}{\mathrm{det}A^{(k-1)}}\tilde{A}^{(k-1)}$ is 
$$\left(
\begin{array}{ccccc}
    (-1)^{k+1}\displaystyle\frac{\mathrm{det}A^{(k-1)}_{k,1}}{\mathrm{det}A^{(k-1)}} \\ \\
    (-1)^{k+2}\displaystyle\frac{\mathrm{det}A^{(k-1)}_{k,2}}{\mathrm{det}A^{(k-1)}} \\ \\
    \vdots \\ \\
    (-1)^{2k}\displaystyle\frac{\mathrm{det}A^{(k-1)}_{k,k}}{\mathrm{det}A^{(k-1)}}
    \end{array}
    \right).$$
    Therefore, we shall show the following equality.
    $$\left(
\begin{array}{ccccc}
    (-1)^{k+1}\displaystyle\frac{\mathrm{det}A^{(k-1)}_{k,1}}{\mathrm{det}A^{(k-1)}} \\ \\
    (-1)^{k+2}\displaystyle\frac{\mathrm{det}A^{(k-1)}_{k,2}}{\mathrm{det}A^{(k-1)}} \\ \\
    \vdots \\ \\
    (-1)^{2k}\displaystyle\frac{\mathrm{det}A^{(k-1)}_{k,k}}{\mathrm{det}A^{(k-1)}}
    \end{array}
    \right)=
    \left(
    \begin{array}{cccc}
    \displaystyle\frac{1}{(0!)^{2}}\cdot \displaystyle\frac{1}{(k-1)!} \\ \\
    \displaystyle\frac{1}{(1!)^{2}}\cdot \displaystyle\frac{1}{(k-2)!} \\ \\
    \vdots \\ \\
    \displaystyle\frac{1}{((k-1)!)^{2}}\cdot \displaystyle\frac{1}{0!}
    \end{array}
    \right)(=: {\bm b}_{k}).$$ 
    Let ${\bm a}_{m}$ be the $m$ th row of $A^{(k-1)}$. Since the $k$ th column vector ${\bm c}$ of $(A^{(k-1)})^{-1}$ is characterized by the conditions: ${\bm a_{m}}{\bm c}=\delta_{m,k}\ \ (1\le m\le k)$, for a proof of theorem 3.1 we will show the followings:
    $${\bm a}_{m}{\bm b}_{k}=\delta_{m,k};$$
 $$\mathrm{det}A^{(k-1)}\neq 0,$$
    where $\delta_{m,k}$ is  the Kronecker delta.\\
    \ \ \ \ First we show ${\bm a}_{m}{\bm b}_{k}=\delta_{m,k}$. Since ${\bm a}_{m}=((-1)^{m-1}(m-1)!, \cdots, (-1)^{m+k-2}(m+k-2)!)$, we have that  $${\bm a}_{m}{\bm b}_{k}=\sum_{i=1}^{k} (-1)^{m+i}(m+i-2)!\frac{1}{((i-1)!)^{2}}\cdot \frac{1}{(k-i)!}.$$
Putting $a_{i}:=(-1)^{m+i}(m+i-2)!\displaystyle\frac{1}{((i-1)!)^{2}}\cdot \displaystyle\frac{1}{(k-i)!}$, they satisfy the following recurrence relation:
$$a_{i+1}=(-1)\left(\frac{k}{i}-1\right) \left(\frac{m-1}{i}+1\right)a_{i}\ \ \ \ (1\le i\le k).$$
Therefore, $${\bm a}_{m}{\bm b}_{k}=(-1)^{m+1}\displaystyle\frac{(m-1)!}{(k-1)!}\left\{1+\sum_{i=1}^{k-1} (-1)^{i}\left(\frac{k}{1}-1\right)\cdots \left(\frac{k}{i}-1\right) \left(\frac{m-1}{1}+1\right) \cdots \left(\frac{m-1}{i}+1\right) \right\}.$$
We put $$f_{k-1}(x,y):=1+\sum_{i=1}^{k-1} (-1)^{i}\left(\frac{x}{1}-1\right)\cdots \left(\frac{x}{i}-1\right) \left(\frac{y}{1}+1\right) \cdots \left(\frac{y}{i}+1\right).$$
We now prove the following equalities. 
$$f_{k-1}(k,m-1)=
  \begin{cases}
    \ \ \ 0\ \ \ \ \ \ \ \ \ \  \mathrm{if}\ \ 1\le m\le k-1, \\
    (-1)^{k-1}\ \ \ \ \mathrm{if}\ \ m=k.
  \end{cases}$$
Since $$\left(\displaystyle\frac{k}{1}-1\right)\cdots \left(\displaystyle\frac{k}{i}-1\right)={}_{k-1} C _i\ $$and \vspace{1mm}
  $$\left(\displaystyle\frac{m-1}{1}+1\right) \cdots \left(\displaystyle\frac{m-1}{i}+1\right)={}_{m+i-1} C _i,$$
We have that  $$f_{k-1}(k,m-1)=1+\sum_{i=1}^{k-1} (-1)^{i}{}_{k-1} C _i \cdot {}_{m+i-1} C _i.$$
Since ${}_{m+i-1} C _i= {}_{m+i-1} C _{m-1}=\displaystyle\frac{(m+i-1)\cdots (i+1)}{(m-1)!}$, we obtain that 
$$f_{k-1}(k,m-1)=1+\frac{1}{(m-1)!}\sum_{i=1}^{k-1} (-1)^{i}{}_{k-1} C _i (m+i-1)\cdots (i+1).$$
\ \ \ \ By considering the $m-1$ th derivative of the both sides of the identity $$x^{m-1}(1-x)^{k-1}=\displaystyle\sum_{i=0}^{k-1} (-1)^{i} {}_{k-1} C _i x^{m+i-1},$$ we have that
$$\sum_{j=0}^{m-1} {}_{m-1} C _j (x^{m-1})^{(j)}((1-x)^{k-1})^{(m-1-j)} =(m-1)!+\sum_{i=1}^{k-1} (-1)^{i} {}_{k-1} C _i (m+i-1)\cdots (i+1)x^{i}.$$
If $1\le m \le k-1$, then 
$$((1-x)^{k-1})^{(m-1-j)}|_{x=1}=0\ \ \ \ (0\le j\le m-1).$$
So by substituting 1 for $x$ on the both sides, we have $$\displaystyle\sum_{i=1}^{k-1}(-1)^{i} {}_{k-1} C _i (m+i-1)\cdots (i+1)=-(m-1)!.$$ Therefore $$f_{k-1}(k,m-1)=1+\frac{1}{(m-1)!}(-(m-1)!)=0.$$
If $m=k$, then 
$$((1-x)^{k-1})^{(m-1-j)}|_{x=1}=\begin{cases}\ \ \ 0\ \ \ \ \ \ \ \ \ \ \mathrm{if}\ \ 1\le j\le m-1, \\
(-1)^{k-1}(k-1)!\ \ \ \ \mathrm{if}\ \ j=0.
\end{cases}$$ Therefore we have that $$\displaystyle\sum_{i=1}^{k-1} (-1)^{i} {}_{k-1} C _i (k+i-1)\cdots (i+1)=(k-1)!((-1)^{k-1}-1).$$ Hence $$f_{k-1}(k,k-1)=1+\frac{1}{(k-1)!}\cdot (k-1)!((-1)^{k-1}-1)=(-1)^{k-1}.$$
Therefore ${\bm a}_{m}{\bm b}_{k}=\delta_{m,k}$.\\
\ \ \ \ Next we prove $\mathrm{det}A^{(k-1)}\neq 0$ by induction on $k$ by using ${\bm a}_{m}{\bm b}_{k}=\delta_{m,k}\ \ (1\le m\le k)$. If $k=2$, \vspace{5mm} $A^{(1)}=\left(
    \begin{array}{cc}
     1 & -1   \\
      -1 & 2  
    \end{array}
  \right)$, therefore
   $$\mathrm{det}A^{(1)}=1\neq 0.$$
Assume that $\mathrm{det}A^{(k-2)}\neq 0$ for $k\ge3$. We will show that ${\bm a}_{1},\cdots, {\bm a}_{k}$ are linearly independent.
Assume that $x_{1}{\bm a}_{1}+\cdots +x_{k}{\bm a}_{k}={\bm 0}$. Then we have that $$x_{1}\cdot 0+\cdots x_{k-1}\cdot 0+x_{k}\cdot 1=0$$ by multiplying the both sides by ${\bm b}_{k}$ from right. Therefore $x_{k}=0$.\\
\ \ \ \ Let ${\bm a}^{'}_{m}$ be the $m$ th row vector of $A^{(k-2)}$. Since $x_{k}=0$, $$x_{1}{\bm a}_{1}+\cdots +x_{k-1}{\bm a}_{k-1}={\bm 0}.$$
In particular, it follow that $$x_{1}{\bm a}^{'}_{1}+\cdots +x_{k-1}{\bm a}^{'}_{k-1}={\bm 0}.$$ By assumption, ${\bm a}^{'}_{1},\cdots, {\bm a}^{'}_{k-1}$ is linearly independent. Therefore $$x_{1}=x_{2}=\cdots =x_{k-1}=0.$$ Hence ${\bm a}_{1}, \cdots, {\bm a}_{k}$ are linearly independent, i.e., $\mathrm{det}A^{(k-1)}\neq 0.$
\end{proof}
\vspace{5mm}
Since $A_{k,k}^{(k-1)}=A^{(k-2)}$, Theorem3.1 implies $$\mathrm{det}A^{(k-1)}=((k-1)!(k-2)!\cdots 2!\cdot 1!)^{2}.$$
\ \ \ \ We remark that the determinant of $A^{(k-1)}$ is the Hankel determinant $H_{k}^{(0)}$. Hence we obtain that
$$H_{k}^{(0)}=\mathrm{det}A^{(k-1)}=((k-1)!(k-2)!\cdots 2!\cdot 1!)^{2}.$$
Also the determinant of $A^{(k-1)}_{k,1}$ is the Hankel determinant $H_{k-1}^{(1)}$. Using the identity obtained by setting $m=1$ in Theorem3.1, we have that
$$H_{k-1}^{(1)}=\mathrm{det}A_{k,1}^{(k-1)}=(-1)^{k+1}\frac{1}{(k-1)!}\mathrm{det}A^{(k-1)}=(-1)^{k+1}(k-1)!((k-2)!(k-3)!\cdots 2!\cdot 1!)^{2}.$$
\section{Proof of Theorem2.1}
\ \ \ \ In this section, we will prove Theorem2.1.\\
\ \ \ \ We prove Theorem2.1 by induction. If $m=1$, we have seen that
$$F_{1}(x)\sim x,\ F_{2}(x)\sim 1\ \ \ \ (x\to \infty).$$
Let $k$ be an integer grater than or equal to 2. We assume that for any integer $m$ with $2\le m\le k$,
$$F_{2m-1}(x)\sim x,\ F_{2m}(x)\sim \displaystyle\frac{1}{m}\ \ \ \ (x\to \infty).$$
Using linear fractional transformations, these assumptions (other than $m=k$) are expressed as  that for above $m$
$$F_{2m-1}(x)=\left(
    \begin{array}{cc}
     0 & 1   \\
      1 & -\displaystyle\frac{1}{m-1}  
    \end{array}
  \right)F_{2m-2}(x),\ \ F_{2m-2}(x)=\left(
    \begin{array}{cc}
     0 & 1   \\
      1 & -x 
    \end{array}
  \right)F_{2m-3}(x).$$
  Hence,
   \begin{eqnarray*}F_{2m-1}(x)&=&\left(
    \begin{array}{cc}
     0 & 1   \\
      1 & -\displaystyle\frac{1}{m-1}  
    \end{array}
  \right)\left(
    \begin{array}{cc}
     0 & 1   \\
      1 & -x 
    \end{array}
  \right)F_{2m-3}(x) \\ \\ \\ 
  &=&\left(
    \begin{array}{cc}
     1 & -x   \\
      -\displaystyle\frac{1}{m-1} & 1+\displaystyle\frac{x}{m-1} 
    \end{array}
  \right)F_{2m-3}(x).\end{eqnarray*}
  Therefore, we obtain that
  \begin{eqnarray*}
  F_{2k-1}(x)&=&\left(
    \begin{array}{cc}
     1 & -x   \\
      -\displaystyle\frac{1}{k-1} & 1+\displaystyle\frac{x}{k-1} 
    \end{array}
  \right)\cdots \left(
    \begin{array}{cc}
     1 & -x   \\
      -\displaystyle\frac{1}{2} & 1+\displaystyle\frac{x}{2} 
    \end{array}
  \right)\left(
    \begin{array}{cc}
     1 & -x   \\
      -1 & 1+x 
    \end{array}
  \right)F_{1}(x)\\ \\ \\
  &=& \left(
    \begin{array}{cc}
     1 & -x   \\
      -\displaystyle\frac{1}{k-1} & 1+\displaystyle\frac{x}{k-1} 
    \end{array}
  \right)\cdots \left(
    \begin{array}{cc}
     1 & -x   \\
      -\displaystyle\frac{1}{2} & 1+\displaystyle\frac{x}{2} 
    \end{array}
  \right)\left(
    \begin{array}{cc}
     1 & -x   \\ 
      -1 & 1+x 
    \end{array}
  \right)\displaystyle\frac{1}{F(x)}.
  \end{eqnarray*}
  \ \ \ \ We put 
  $$D_{k-1}:=\left(
    \begin{array}{cc}
     1 & -x   \\
      -\displaystyle\frac{1}{k-1} & 1+\displaystyle\frac{x}{k-1} 
    \end{array}
  \right)\cdots \left(
    \begin{array}{cc}
     1 & -x   \\
      -\displaystyle\frac{1}{2} & 1+\displaystyle\frac{x}{2} 
    \end{array}
  \right)\left(
    \begin{array}{cc}
     1 & -x   \\ 
      -1 & 1+x 
    \end{array}
  \right).$$
  Here, we define $D_{0}:=\left(\begin{array}{cc}
     1 & 0   \\ 
      0 & 1 
    \end{array}\right)$ since $F_{1}(x)=\left(\begin{array}{cc}
     1 & 0   \\
      0 & 1
    \end{array}
  \right)\displaystyle\frac{1}{F(x)}$.
  We also set 
  $$D_{k-1}=\left(
    \begin{array}{cc}
     p_{k-1}(x) & q_{k-1}(x)   \\
      r_{k-1}(x) & s_{k-1}(x) 
    \end{array}
  \right).$$
  Then, we get that
   $$\left(
    \begin{array}{cc}
     p_{k}(x) & q_{k}(x)   \\
      r_{k}(x) & s_{k}(x) 
    \end{array}
  \right)=D_{k}=\left(
    \begin{array}{cc}
     1 & -x   \\
      -\displaystyle\frac{1}{k} & 1+\displaystyle\frac{x}{k} 
    \end{array}
  \right)\left(
    \begin{array}{cc}
     p_{k-1}(x) & q_{k-1}(x)  \\
      r_{k-1}(x) & s_{k-1}(x) 
    \end{array}
  \right)$$
  $$=\left(
  \begin{array}{cc}
  p_{k-1}(x)-xr_{k-1}(x) & q_{k-1}(x)-xs_{k-1}(x) \\ \vspace{3mm}
  -\displaystyle\frac{1}{k}p_{k-1}(x)+\left(1+\displaystyle\frac{x}{k}\right)r_{k-1}(x) &  -\displaystyle\frac{1}{k}q_{k-1}(x)+\left(1+\displaystyle\frac{x}{k}\right)s_{k-1}(x)
  \end{array}\right) .$$
  Comparing the (2,1)th entries and solving for $p_{k-1}(x)$, we have that $p_{k-1}(x)=-kr_{k}(x)+(k+x)r_{k-1}(x)$. Substituting this into the equation obtained by comparing the (1,1)th entries, we get that $$p_{k}(x)=-k(r_{k}(x)-r_{k-1}(x)).$$
  Similarly we have that 
  $$q_{k}(x)=-k(s_{k}(x)-s_{k-1}(x)).$$
  Therefore,
  $$D_{k}=\left(
    \begin{array}{cc}
     -k(r_{k}(x)-r_{k-1}(x)) &  -k(s_{k}(x)-s_{k-1}(x))  \\
      r_{k}(x) & s_{k}(x) 
    \end{array}
  \right).$$
  Moreover substituting $p_{k-1}(x)=-(k-1)(r_{k-1}(x)-r_{k-2}(x))$ (resp. $q_{k-1}(x)=-(k-1)(s_{k-1}(x)-s_{k-2}(x))$) into the equation obtained by comparing the (2,1)th (resp. (2,2)th) entries, we have the following recurrence relations:
   $$ \begin{cases}
    r_{k}(x)=\displaystyle\frac{2k-1+x}{k}r_{k-1}(x)-\displaystyle\frac{k-1}{k}r_{k-2}(x), \\ \\
   s_{k}(x)=\displaystyle\frac{2k-1+x}{k}s_{k-1}(x)-\displaystyle\frac{k-1}{k}s_{k-2}(x),
  \end{cases}$$
  with initial conditions: $(r_{0}(x), r_{1}(x))=(0, -1)$, $(s_{0}(x), s_{1}(x))=(1, 1+x)$. We see that the degree of $r_{k}(x)$ and $s_{k}(x)$ are $k-1$ and $k$, respectively. \\
  \ \ \ \ By our definition and assumption: $F_{2k}(x)\sim \displaystyle\frac{1}{k}$\ \ ($x\to \infty$), we obtain that
  $$F_{2k+1}(x)=\displaystyle\frac{1}{F_{2k}-\displaystyle\frac{1}{k}}=\left(
    \begin{array}{cc}
     1 & -x   \\
      -\displaystyle\frac{1}{k} & 1+\displaystyle\frac{x}{k} 
    \end{array}
  \right)F_{2k-1}(x).$$ Then we have that
  $$F_{2k+1}(x)=D_{k}\displaystyle\frac{1}{F(x)}=\displaystyle\frac{-k(r_{k}(x)-r_{k-1}(x))-k(s_{k}(x)-s_{k-1}(x))F(x)}{r_{k}(x)+s_{k}(x)F(x)}.\eqno(4.1)$$
  Dividing both sides of (4.1) by $x$,  we get the following:
   $$\displaystyle\frac{F_{2k+1}(x)}{x}=\displaystyle\frac{-k(r_{k}(x)-r_{k-1}(x))-k(s_{k}(x)-s_{k-1}(x))F(x)}{xr_{k}(x)+xs_{k}(x)F(x)}.$$
   Here, using the identity in Proposition 2.1 obtained by setting $n=2k+1$ (resp. $n=2k$) in the denominator (resp. the numerator), we have that 
   $$\displaystyle\frac{F_{2k+1}(x)}{x}=\displaystyle\frac{-kx^{k}(r_{k}(x)-r_{k-1}(x))-kx^{k}(s_{k}(x)-s_{k-1}(x))\left(\displaystyle{\sum_{l=1}^{2k} (-1)^{l-1}(l-1)!\displaystyle\frac{1}{x^{l}}}+\mathcal{O}\left(\displaystyle\frac{1}{x^{2k+1}}\right)\right)}
 {x^{k+1}r_{k}(x)+x^{k+1}s_{k}(x)\left(\displaystyle{\sum_{l=1}^{2k+1} (-1)^{l-1}(l-1)!\displaystyle\frac{1}{x^{l}}}+\mathcal{O}\left(\displaystyle\frac{1}{x^{2k+2}}\right)\right)}.$$
 Since $$x^{j}\mathcal{O}\left(\displaystyle\frac{1}{x^{2k+1}}\right)=x^{j-(2k+1)}\cdot \displaystyle\frac{\mathcal{O}\left(\displaystyle\frac{1}{x^{2k+1}}\right)}{\displaystyle\frac{1}{x^{2k+1}}} \to 0\ \ \ \ \ \ (x\to\infty)$$ for $k \le j \le 2k$ and $$x^{j}\mathcal{O}\left(\displaystyle\frac{1}{x^{2k+2}}\right)\to 0\ \ \ \ \ \ (x\to \infty)$$ for $k+1\le j \le 2k+1$, we have that when $x\to \infty$,
 $$kx^{k}(s_{k}(x)-s_{k-1}(x))\mathcal{O}\left(\displaystyle\frac{1}{x^{2k+1}}\right)\to0,$$
  $$x^{k+1}s_{k}(x)\mathcal{O}\left(\displaystyle\frac{1}{x^{2k+2}}\right)\to0.$$
Therefore to show that $F_{2k+1}(x)\sim x\ \ (x\to \infty)$, it is sufficient to prove the following two congruence relations modulo $I:=\displaystyle\frac{1}{x}\mathbb{Q}\left[\displaystyle\frac{1}{x}\right]$:
$$
  -kx^{k}(r_{k}(x)-r_{k-1}(x))-kx^{k}(s_{k}(x)-s_{k-1}(x))\left(\displaystyle{\sum_{l=1}^{2k} (-1)^{l-1}(l-1)!\displaystyle\frac{1}{x^{l}}}\right)\equiv k! \ \ (\mathrm{mod}\ I),  \eqno(4.2)$$
 $$x^{k+1}r_{k}(x)+x^{k+1}s_{k}(x)\left(\displaystyle{\sum_{l=1}^{2k+1} (-1)^{l-1}(l-1)!\displaystyle\frac{1}{x^{l}}}\right)\equiv k!\ \ (\mathrm{mod}\ I). \eqno(4.3)
  $$
  \ \ \ \ We put $$r_{k}(x):=\sum_{k^{'}=0}^{k-1} R_{k^{'}}^{(k)}x^{k^{'}},\ s_{k}(x):=\sum_{k^{'}=0}^{k} S_{k^{'}}^{(k)}x^{k^{'}}\ \ \ \ (R_{k^{'}}^{(k)},\ S_{k^{'}}^{(k)}\in \mathbb{Q}).$$
  Then substituting them into the left hand side of (4.3), we have that
  $$\sum_{k^{'}=0}^{k-1} R_{k^{'}}^{(k)}x^{k^{'}+k+1}+\sum_{k^{'}=0}^{k} S_{k^{'}}^{(k)}x^{k^{'}+k+1}\left(\displaystyle{\sum_{l=1}^{2k+1} (-1)^{l-1}(l-1)!\displaystyle\frac{1}{x^{l}}}\right)\equiv k!.\ \ \ \ (\mathrm{mod} I)$$
  By calculating the coefficient of $x^{n}\ (n=2k, \cdots, 1, 0)$ on the left hand side, we see that this is equivalent to the following conditions:
 $$
  R_{n-k-1}^{(k)}+\sum_{l=1}^{2k+1-n} S_{l+n-k-1}^{(k)} (-1)^{l-1}(l-1)!=0\ \ \ \ (n=k+1, k+2, \cdots, 2k), \eqno(4.4)$$ 
 $$ \sum_{l=k-n+1}^{2k+1-n} S_{l+n-k-1}^{(k)}(-1)^{l-1}(l-1)!=0\ \ \ \ (n=1,2,\cdots, k), \eqno(4.5)$$ 
  $$\sum_{l=k+1}^{2k+1} S_{l-k-1}^{(k)}(-1)^{l-1}(l-1)!=k!.\eqno(4.6)$$
  We will show (4.4), (4.5) and (4.6) by induction on $k$. It is straightforward to check them for $k=1,2$. We assume (4.4), (4.5) and (4.6) for $k-1$ and $k-2$. Hence it holds that
  $$R_{n-k}^{(k-1)}+\sum_{l=1}^{2k-1-n} S_{l+n-k}^{(k-1)} (-1)^{l-1}(l-1)!=0\ \ \ \ (n=k, k+1, \cdots, 2k-2),$$
  $$R_{n-k+1}^{(k-2)}+\sum_{l=1}^{2k-3-n} S_{l+n-k+1}^{(k-2)} (-1)^{l-1}(l-1)!=0\ \ \ \ (n=k-1, k, \cdots, 2k-4),$$
  $$\sum_{l=k-n}^{2k-1-n} S_{l+n-k}^{(k-1)}(-1)^{l-1}(l-1)!=0\ \ \ \ (n=1,2,\cdots, k-1),$$
  $$\sum_{l=k-n-1}^{2k-3-n} S_{l+n-k+1}^{(k-2)}(-1)^{l-1}(l-1)!=0\ \ \ \ (n=1,2,\cdots, k-2),$$
  $$\sum_{l=k}^{2k-1} S_{l-k}^{(k-1)}(-1)^{l-1}(l-1)!=(k-1)!,$$
    $$\sum_{l=k-1}^{2k-3} S_{l-k+1}^{(k-2)}(-1)^{l-1}(l-1)!=(k-2)!.$$
    \ \ \ \ By recurrence relation, it holds that
  $$\sum_{k^{'}=0}^{k-1} R_{k^{'}}^{(k)}x^{k^{'}}=\frac{x+2k-1}{k}\sum_{k^{'}=0}^{k-2} R_{k^{'}}^{(k-1)}x^{k^{'}}-\frac{k-1}{k}\sum_{k^{'}=0}^{k-3} R_{k^{'}}^{(k-2)}x^{k^{'}}.$$
 By comparing the coefficients, we have that
 $$R_{k-1}^{(k)}=\frac{1}{k}R_{k-2}^{(k-1)},\ R_{k-2}^{(k)}=\frac{2k-1}{k}R_{k-2}^{(k-1)}+\frac{1}{k}R_{k-3}^{(k-1)},$$\vspace{0.5mm}
  $$R_{k^{'}}^{(k)}=\frac{2k-1}{k}R_{k^{'}}^{(k-1)}+\frac{1}{k}R_{k^{'}-1}^{(k-1)}-\frac{k-1}{k}R_{k^{'}}^{(k-2)}\ \ \ \ (k^{'}=1,2,\cdots,k-3),$$\vspace{0.5mm}
  $$R_{0}^{(k)}=\frac{2k-1}{k}R_{0}^{(k-1)}-\frac{k-1}{k}R_{0}^{(k-2)}.$$
  Similarly, we have that
    $$S_{k}^{(k)}=\frac{1}{k}S_{k-1}^{(k-1)},\ S_{k-1}^{(k)}=\frac{2k-1}{k}S_{k-1}^{(k-1)}+\frac{1}{k}S_{k-2}^{(k-1)},$$\vspace{0.5mm}
  $$S_{k^{'}}^{(k)}=\frac{2k-1}{k}S_{k^{'}}^{(k-1)}+\frac{1}{k}S_{k^{'}-1}^{(k-1)}-\frac{k-1}{k}S_{k^{'}}^{(k-2)}\ \ \ \ (k^{'}=1,2,\cdots,k-2),$$\vspace{0.5mm}
  $$S_{0}^{(k)}=\frac{2k-1}{k}S_{0}^{(k-1)}-\frac{k-1}{k}S_{0}^{(k-2)}.$$
  First we show (4.4) for $k$.
If $n=2k$, the left hand side of (4.4) is equal to
  $$R_{k-1}^{(k)}+S_{k}^{(k)}=\frac{1}{k}R_{k-2}^{(k-1)}+\frac{1}{k}S_{k-1}^{(k-1)}=\frac{1}{k}(R_{k-2}^{(k-1)}+S_{k-1}^{(k-1)}) =0,$$
  and if $n=2k-1$, that is equal to
   \begin{eqnarray*}
   R_{k-2}^{(k)}+S_{k-1}^{(k)}-S_{k}^{(k)}&=&\frac{2k-1}{k}R_{k-2}^{(k-1)}+\frac{1}{k}R_{k-3}^{(k-1)}+\frac{2k-1}{k}S_{k-1}^{(k-1)}+\frac{1}{k}S_{k-2}^{(k-1)}-\frac{1}{k}S_{k-1}^{(k-1)} \\ \\
   &=&\frac{2k-1}{k}(R_{k-2}^{(k-1)}+S_{k-1}^{(k-1)})+\frac{1}{k}(R_{k-3}^{(k-1)}-S_{k-1}^{(k-1)}+S_{k-2}^{(k-1)})\\ \\
   &=&0.
   \end{eqnarray*}
   If $k+2\le n\le 2k-2$, we have that  $$R_{n-k-1}^{(k)}=\frac{2k-1}{k}R_{n-k-1}^{(k-1)}+\frac{1}{k}R_{n-k-2}^{(k-1)}-\frac{k-1}{k}R_{n-k-1}^{(k-2)}$$
and  \begin{eqnarray*}
\sum_{l=1}^{2k+1-n}S_{l+n-k-1}^{(k)}(-1)^{l-1}(l-1)!=S_{k}^{(k)}(-1)^{2k-n}(2k-n)!&+&S_{k-1}^{(k)}(-1)^{2k-n-1}(2k-n-1)!\\
&+&\sum_{l=1}^{2k-n-1}S_{l+n-k-1}^{(k)}(-1)^{l-1}(l-1)!
\end{eqnarray*}
$$=\frac{1}{k}S_{k-1}^{(k-1)}(-1)^{2k-n}(2k-n)!+\left(\frac{2k-1}{k}S_{k-1}^{(k-1)}+\frac{1}{k}S_{k-2}^{(k-1)}\right)(-1)^{2k-n-1}(2k-n-1)!$$
$$+\sum_{l=1}^{2k-n-1}\left(\frac{2k-1}{k}S_{l+n-k-1}^{(k-1)}+\frac{1}{k}S_{l+n-k-2}^{(k-1)}-\frac{k-1}{k}S_{l+n-k-1}^{(k-2)}\right)(-1)^{l-1}(l-1)!.$$
Therefore, the left hand side of (4.4) is equal to
$$\frac{2k-1}{k}\left(R_{n-k-1}^{(k-1)}+\sum_{l=1}^{2k-n}S_{l+n-k-1}^{(k-1)}(-1)^{l-1}(l-1)!\right)+\frac{1}{k}\left(R_{n-k-2}^{(k-1)}+\sum_{l=1}^{2k-n+1}S_{l+n-k-2}^{(k-1)}(-1)^{l-1}(l-1)!\right)$$
   $$-\frac{k-1}{k}\left(R_{n-k-1}^{(k-2)}+\sum_{l=1}^{2k-n-1}S_{l+n-k-1}^{(k-2)}(-1)^{l-1}(l-1)!\right)=0.$$
   If $n=k+1$, the left hand side of (4.4) is equal to
    $$R_{0}^{(k)}+S_{k}^{(k)}(-1)^{k-1}(k-1)!+S_{k-1}^{(k)}(-1)^{k-2}(k-2)!+\sum_{l=1}^{k-2}S_{l}^{(k)}(-1)^{l-1}(l-1)!$$
    $$=\frac{2k-1}{k}R_{0}^{(k-1)}-\frac{k-1}{k}R_{0}^{(k-2)}+\frac{1}{k}S_{k-1}^{(k-1)}(-1)^{k-1}(k-1)!+\left(\frac{2k-1}{k}S_{k-1}^{(k-1)}+\frac{1}{k}S_{k-2}^{(k-1)}\right)(-1)^{k-2}(k-2)!$$
    $$+\sum_{l=1}^{k-2}\left(\frac{2k-1}{k}S_{l}^{(k-1)}+\frac{1}{k}S_{l-1}^{(k-1)}-\frac{k-1}{k}S_{l}^{(k-2)}\right)(-1)^{l-1}(l-1)!.$$
    Hence, that is equal to $$\frac{2k-1}{k}\left(R_{0}^{(k-1)}+\sum_{l=1}^{k-1}S_{l}^{(k-1)}(-1)^{l-1}(l-1)!\right)-\frac{k-1}{k}\left(R_{0}^{(k-2)}+\sum_{l=1}^{k-2}S_{l}^{(k-2)}(-1)^{l-1}(l-1)!\right)$$
   $$+\frac{1}{k}\sum_{l=1}^{k}S_{l-1}^{(k-1)}(-1)^{l-1}(l-1)!=0.$$
These complete a proof of (4.4) for $k$. \\
\ \ \ \ Next we will show that (4.5) and (4.6). We rewrite (4.5) and (4.6) the followings, respectively:
$$ \sum_{l=0}^{k}S_{l}^{(k)}(-1)^{l+k-n}(l+k-n)!=0\ \ \ \ (n=1,2,\cdots,k),\eqno(4.7)$$
$$\sum_{l=0}^{k}S_{l}^{(k)}(-1)^{l+k}(l+k)!=k!.\eqno(4.8)$$
Similarly, we rewrite the last four assumptions to followings:
$$ \begin{cases}
\displaystyle\sum_{l=0}^{k-1}S_{l}^{(k-1)}(-1)^{l+k-1-n}(l+k-1-n)!=0,\ \ \ \ (n=k-1,k-2,\cdots,1)\\ \\
\displaystyle\sum_{l=0}^{k-1}S_{l}^{(k-1)}(-1)^{l+k-1}(l+k-1)!=(k-1)!,
\end{cases} \eqno(4.9)$$
$$\begin{cases}
\displaystyle\sum_{l=0}^{k-2}S_{l}^{(k-2)}(-1)^{l+k-2-n}(l+k-2-n)!=0,\ \ \ \ (n=k-2,k-3,\cdots,1) \\ \\
\displaystyle\sum_{l=0}^{k-2}S_{l}^{(k-2)}(-1)^{l+k-2}(l+k-2)!=(k-2)!.
\end{cases} \eqno(4.10)$$
We consider the matrix representation of (4.9):
$$\left(
\begin{array}{ccccc}
     (-1)^{0}\cdot 0! & (-1)^{1}\cdot 1! & \cdots & (-1)^{k-2}\cdot (k-2)! & (-1)^{k-1}\cdot (k-1)!   \vspace{2mm} \\
     (-1)^{1}\cdot 1!  &  (-1)^{2}\cdot 2! & \cdots &  (-1)^{k-1}\cdot (k-1)! & (-1)^{k}\cdot k!  \vspace{2mm} \\
     \vdots & \vdots & \ddots & \vdots & \vdots \\ \\
     (-1)^{k-2}\cdot (k-2)! & (-1)^{k-1}\cdot (k-1)! & \cdots & (-1)^{2k-4}\cdot (2k-4)! & (-1)^{2k-3}\cdot (2k-3)! \vspace{2mm} \\
     (-1)^{k-1}\cdot (k-1)! & (-1)^{k}\cdot k! & \cdots & (-1)^{2k-3}\cdot (2k-3)! & (-1)^{2k-2}\cdot (2k-2)!
    \end{array}
    \right)
    \left(\begin{array}{ccccc}
    S_{0}^{(k-1)}   \vspace{2mm} \\
    S_{1}^{(k-1)}  \vspace{2mm} \\
    \vdots \vspace{2mm} \\
    S_{k-2}^{(k-1)} \vspace{2mm} \\
    S_{k-1}^{(k-1)}
    \end{array}\right)$$
    $$=\left(\begin{array}{ccccc}
    0 \\
    0 \\
    \vdots \\
    0 \\
    (k-1)!
    \end{array}\right).$$
By our definition of $A^{(k-1)}$, we obtain that
    $$A^{(k-1)}\left(
    \begin{array}{ccccc}
    S_{0}^{(k-1)}\vspace{2mm} \\
    S_{1}^{(k-1)}\vspace{2mm} \\
    \vdots \vspace{2mm} \\
    S_{k-2}^{(k-1)} \vspace{2mm} \\ 
    S_{k-1}^{(k-1)}
    \end{array}\right)=\left(
    \begin{array}{ccccc}
    0\vspace{2mm} \\
    0\vspace{2mm} \\
    \vdots \vspace{2mm} \\
    0 \vspace{2mm} \\
   (k-1)! \end{array}\right).$$
   Similarly, we obtain that
   $$A^{(k-2)}\left(
    \begin{array}{ccccc}
    S_{0}^{(k-2)}\vspace{2mm} \\
    S_{1}^{(k-2)}\vspace{2mm} \\ 
    \vdots \vspace{2mm} \\
    S_{k-3}^{(k-2)} \vspace{2mm} \\
    S_{k-2}^{(k-2)}
    \end{array}\right)=\left(
    \begin{array}{ccccc}
    0\vspace{2mm} \\
    0\vspace{2mm} \\ 
    \vdots \vspace{2mm} \\
    0 \vspace{2mm} \\ 
   (k-2)! \end{array}\right).$$
   Therefore we have that
    $$S_{l}^{(k-1)}=(-1)^{k+l+1}(k-1)!\frac{\mathrm{det}A_{k,l+1}^{(k-1)}}{\mathrm{det}A^{(k-1)}}=\frac{{}_{k-1} C _l}{l!}\ \ \ \ (l=0,1,\cdots,k-1), \eqno(4.11)$$
   $$S_{l}^{(k-2)}=(-1)^{k+l}(k-2)!\frac{\mathrm{det}A_{k-1,l+1}^{(k-2)}}{\mathrm{det}A^{(k-2)}}=\frac{{}_{k-2} C _l}{l!} \ \ \ \ (l=0,1,\cdots,k-2),\eqno(4.12)$$
   by Cramer's rule and Theorem 3.1. Now we will show (4.8). We rewrite the left hand side of (4.8) by the recurrence relations:
   $$\frac{1}{k}S_{k-1}^{(k-1)}(-1)^{2k}(2k)!+\left(\frac{2k-1}{k}S_{k-1}^{(k-1)}+\frac{1}{k}S_{k-2}^{(k-1)}\right)(-1)^{2k-1}(2k-1)!$$
   $$+\sum_{l=1}^{k-2} \left(\frac{2k-1}{k}S_{l}^{(k-1)}+\frac{1}{k}S_{l-1}^{(k-1)}-\frac{k-1}{k}S_{l}^{(k-2)}\right)(-1)^{l+k}(l+k)!$$
   $$+\left(\frac{2k-1}{k}S_{0}^{(k-1)}-\frac{k-1}{k}S_{0}^{(k-2)}\right)(-1)^{k}k!$$
     $$=\frac{2k-1}{k}\sum_{l=0}^{k-1}S_{l}^{(k-1)}(-1)^{k+l}(k+l)!+\frac{1}{k}\sum_{l=0}^{k-1}S_{l}^{(k-1)}(-1)^{k+l+1}(k+l+1)!$$
     $$-\frac{k-1}{k}\sum_{l=0}^{k-2}S_{l}^{(k-2)}(-1)^{k+l}(k+l)!.\eqno(4.13)$$
     Then substituting (4.11) and (4.12) into (4.13), we have that the left hand side of (4.8) is equal to
    $$\frac{2k-1}{k}\sum_{l=0}^{k-1} {}_{k-1} C _{l}(-1)^{k+l}(k+l)\cdots (l+1)+\frac{1}{k}\sum_{l=0}^{k-1} {}_{k-1} C _l(-1)^{k+l+1}(k+l+1)\cdots (l+1)$$
    $$-\frac{k-1}{k}\sum_{l=0}^{k-2}{}_{k-2} C _l (-1)^{k+l}(k+l)\cdots(l+1). \eqno(4.14)$$
    By considering the $k$ th derivative of the both sides of the identity $(-1)^{k}x^{k}(1-x)^{k-1}=\displaystyle\sum_{l=0}^{k-1} {}_{k-1} C _{l}(-1)^{l+k}x^{l+k},$ we have that
  $$(-1)^{k}\sum_{i=0}^{k} {}_{k} C _i (x^{k})^{(i)}((1-x)^{k-1})^{(k-i)} =\sum_{l=0}^{k-1}{}_{k-1} C _l  (-1)^{l+k} (l+k)\cdots (l+1)x^{l}.$$
  If $i=0$ or $2\le i\le k$, then we obtain that
  $$((1-x)^{k-1})^{(k-i)}|_{x=1}=0.$$
  Therefore we get that
  $$\sum_{l=0}^{k-1}{}_{k-1} C _l  (-1)^{l+k} (l+k)\cdots (l+1)=(-1)^{k}k\cdot k(-1)^{k-1}(k-1)!=-k\cdot k!.\eqno(4.15)$$
 Similarly we have that
 $$\sum_{l=0}^{k-1}{}_{k-1} C _l  (-1)^{l+k+1} (l+k+1)\cdots (l+1)=\frac{(k+1)k(k+1)!}{2},$$
  $$\sum_{l=0}^{k-2}{}_{k-2} C _l  (-1)^{l+k} (l+k)\cdots (l+1)=\frac{k(k-1)k!}{2}.$$
  Therefore substituting them into (4.14), we get that the left hand side of (4.8) is equal to
  $$-\frac{2k-1}{k}k\cdot k!+\frac{1}{k}\cdot \frac{(k+1)k(k+1)!}{2}-\frac{k-1}{k}\cdot \frac{k(k-1)k!}{2}=k!\left(\frac{k^{2}+2k+1}{2}-2k+1-\frac{k^{2}-2k+1}{2}\right)=k!.$$
  This completes a proof of (4.8).\\
 \ \ \ \ We rewrite the left hand side of (4.7) by the recurrence relations:
 $$\frac{2k-1}{k}\sum_{l=0}^{k-1}S_{l}^{(k-1)}(-1)^{k+l-n}(k+l-n)!+\frac{1}{k}\sum_{l=0}^{k-1}S_{l}^{(k-1)}(-1)^{k+l-n+1}(k+l-n+1)!$$
 $$-\frac{k-1}{k}\sum_{l=0}^{k-2}S_{l}^{(k-2)}(-1)^{k+l-n}(k+l-n)!.\eqno(4.16)$$
 If $n=2$, we get that this is equal to
 $$\frac{2k-1}{k}\cdot0+\frac{1}{k}(k-1)!-\frac{k-1}{k}(k-2)!=0$$
 by (4.9) and (4.10). If $3\le n \le k$, we have that (4.16) is equal to
 $$\frac{2k-1}{k}\cdot0+\frac{1}{k}\cdot0-\frac{k-1}{k}\cdot0=0$$
 by (4.9) and (4.10). If $n=1$, we have that (4.16) is equal to
 $$\frac{2k-1}{k}(k-1)!+\frac{1}{k}\sum_{l=0}^{k-1}S_{l}^{(k-1)}(-1)^{k+l}(k+l)!-\frac{k-1}{k}\sum_{l=0}^{k-2}S_{l}^{(k-2)}(-1)^{k+l-1}(k+l-1)!\eqno(4.17)$$
 by (4.9), (4.11) and (4.12). By the same argument to get (4.15), we see that $$\displaystyle\sum_{l=0}^{k-2}{}_{k-2} C _l  (-1)^{l+k-1} (l+k-1)\cdots (l+1)=-(k-1)\cdot (k-1)!.$$ Then
substituting them into (4.17), we have that
$$\frac{2k-1}{k}(k-1)!-\frac{k\cdot k!}{k}+\frac{(k-1)^{2}(k-1)!}{k}=\frac{(k-1)!}{k}(2k-1-k^{2}+(k-1)^{2})=0.$$
These complete a proof of (4.7). Therefore, these complete a proof of the assertion for (4.3).\\
\ \ Next we will show (4.2). The left hand side of (4.2) is equal to 
 $$-kx^{k}\left(r_{k}(x)+s_{k}(x)\sum_{l=1}^{2k}(-1)^{l-1}(l-1)!\frac{1}{x^{l}}\right)+kx^{k}\left(r_{k-1}(x)+s_{k-1}(x)\sum_{l=1}^{2k}(-1)^{l-1}(l-1)!\frac{1}{x^{l}}\right).$$
 Here, using (4.3), we obtain that
  $$x^{k}\left(r_{k}(x)+s_{k}(x)\displaystyle\sum_{l=1}^{2k}(-1)^{l-1}(l-1)!\displaystyle\frac{1}{x^{l}}\right)$$
 $$=\frac{1}{x}\left\{x^{k+1}r_{k}(x)+x^{k+1}s_{k}(x)\left(\sum_{l=1}^{2k+1}(-1)^{l-1}(l-1)!\frac{1}{x^{l}}-(-1)^{2k}(2k)!\frac{1}{x^{2k+1}}\right)\right\}$$
   $$=\frac{1}{x}\left(x^{k+1}r_{k}(x)+x^{k+1}s_{k}(x)\sum_{l=1}^{2k+1}(-1)^{l-1}(l-1)!\frac{1}{x^{l}}-(2k)!\frac{s_{k}(x)}{x^{k}}\right)\equiv 0\ \ \ \ (\mathrm{mod}\ I).$$
   On the other hand, using (4.3) for $k-1$, we obtain that
   $$x^{k}\left(r_{k-1}(x)+s_{k-1}(x)\displaystyle\sum_{l=1}^{2k}(-1)^{l-1}(l-1)!\displaystyle\frac{1}{x^{l}}\right)$$
   $$=x^{k}r_{k-1}(x)+x^{k}s_{k-1}(x)\left((-1)^{2k-1}(2k-1)!\frac{1}{x^{2k}}+\sum_{l=1}^{2k-1}(-1)^{l-1}(l-1)!\frac{1}{x^{l}}\right)$$
   $$=x^{k}r_{k-1}(x)+x^{k}s_{k-1}(x)\sum_{l=1}^{2k-1}(-1)^{l-1}(l-1)!\frac{1}{x^{l}}-(2k-1)!\frac{s_{k-1}(x)}{x^{k}}\equiv (k-1)!\ \ \ \ (\mathrm{mod}\ I)$$
   Therefore, We have that
   $$-kx^{k}(r_{k}(x)-r_{k-1}(x))-kx^{k}(s_{k}(x)-s_{k-1}(x))\left(\displaystyle{\sum_{l=1}^{2k} (-1)^{l-1}(l-1)!\displaystyle\frac{1}{x^{l}}}\right)\equiv k!\ \ \ \ (\mathrm{mod}\ I).$$
  This completes a proof of (4.2). Hence we complete a proof of that
  $$F_{2k+1}(x)\sim x \ \ \ \ (x\to \infty).$$ 
   \ \ By our definition and assumption: $F_{2k+1}(x)\sim x$\ \ $(x\to \infty)$, $F_{2k+2}(x)=\displaystyle\frac{1}{F_{2k+1}-x}$. Then we have that
   $$F_{2k+2}(x)=\left(\begin{array}{cc}
   0 & 1 \\
   1 & -x \\ 
    \end{array}\right)F_{2k+1}(x)=\left(\begin{array}{cc}
   0 & 1 \\
   1 & -x \\ 
    \end{array}\right)D_{k}\frac{1}{F(x)}.$$
 By
  \begin{eqnarray*}\left(\begin{array}{cc}
   0 & 1 \\
   1 & -x \\ 
    \end{array}\right)D_{k}&=& \left(\begin{array}{cc}
   0 & 1 \\
   1 & -x \\ 
    \end{array}\right)
    \left(
    \begin{array}{cc}
     -k(r_{k}(x)-r_{k-1}(x)) &  -k(s_{k}(x)-s_{k-1}(x))  \\
      r_{k}(x) & s_{k}(x) 
    \end{array}
  \right)\\ \\
  &=&\left(\begin{array}{cc}
   r_{k}(x) & s_{k}(x) \\
   -(k+x)r_{k}(x)+kr_{k-1}(x) &   -(k+x)s_{k}(x)+ks_{k-1}(x)
    \end{array}\right),
  \end{eqnarray*}
  we obtain that
   $$F_{2k+2}(x)=\frac{r_{k}(x)+s_{k}(x)F(x)}{-(k+x)r_{k}(x)+kr_{k-1}(x)-\{(k+x)s_{k}(x)-ks_{k-1}(x)\}F(x)}.\eqno(4.18)$$
   Dividing both sides of (4.18) by $\displaystyle\frac{1}{k+1}$, and using the identity obtained by setting $n=2k+2$ (resp. $n=2k+1$) in the denominator (resp. the numerator), we have that
   $$\frac{F_{2k+2}(x)}{\displaystyle\frac{1}{k+1}}$$
  $$=\frac{(k+1)x^{k+1}r_{k}(x)+(k+1)x^{k+1}s_{k}(x)\left(\displaystyle\sum_{l=1}^{2k+1} (-1)^{l-1}(l-1)!\displaystyle\frac{1}{x^{l}}+\mathcal{O}\left(\displaystyle\frac{1}{x^{2k+2}}\right)\right)}{-x^{k+1}\left\{(k+x)r_{k}(x)-kr_{k-1}(x)+\left((k+x)s_{k}(x)-ks_{k-1}(x)\right)\left(\displaystyle\sum_{l=1}^{2k+2} (-1)^{l-1}(l-1)!\displaystyle\frac{1}{x^{l}}+\mathcal{O}\left(\displaystyle\frac{1}{x^{2k+3}}\right)\right)\right\}}.$$
  Therefore to show that $F_{2k+2}(x)\sim \displaystyle\frac{1}{k+1}\ \ (x\to \infty)$, it is sufficient to prove tha following two congruence relations modulo $I=\displaystyle\frac{1}{x}\mathbb{Q}\left[\displaystyle\frac{1}{x}\right]$:
  $$ (k+1)x^{k+1}r_{k}(x)+(k+1)x^{k+1}s_{k}(x)\left(\displaystyle\sum_{l=1}^{2k+1} (-1)^{l-1}(l-1)!\displaystyle\frac{1}{x^{l}}\right)\equiv (k+1)! \ \ (\mathrm{mod}\ I),\eqno(4.19)$$
  $$-x^{k+1}\left\{(k+x)r_{k}(x)-kr_{k-1}(x)+\left((k+x)s_{k}(x)-ks_{k-1}(x)\right)\left(\displaystyle\sum_{l=1}^{2k+2} (-1)^{l-1}(l-1)!\displaystyle\frac{1}{x^{l}}\right)\right\}$$
  $$\equiv (k+1)!\ \ (\mathrm{mod}\ I).\eqno(4.20)$$
  By (4.3), we see that
  $$(k+1)\left(x^{k+1}r_{k}(x)+x^{k+1}s_{k}(x)\left(\displaystyle{\sum_{l=1}^{2k+1} (-1)^{l-1}(l-1)!\displaystyle\frac{1}{x^{l}}}\right)\right)\equiv (k+1)!\ \ \ \ (\mathrm{mod}\ I).$$
  Therefore this completes a proof of (4.19). Next we will show (4.20). The left side of (4.20) is equal to
  $$-k\left(x^{k+1}r_{k}(x)+x^{k+1}s_{k}(x)\sum_{l=1}^{2k+2} (-1)^{l-1}(l-1)!\displaystyle\frac{1}{x^{l}}\right)-x\left(x^{k+1}r_{k}(x)+x^{k+1}s_{k}(x)\sum_{l=1}^{2k+2} (-1)^{l-1}(l-1)!\displaystyle\frac{1}{x^{l}}\right)$$
  $$+kx\left(x^{k}r_{k-1}(x)
  +x^{k}s_{k-1}(x)\sum_{l=1}^{2k+2} (-1)^{l-1}(l-1)!\displaystyle\frac{1}{x^{l}}\right).\eqno(4.21)$$
  Here, the coefficient of $x^{-1}$ of $x^{k+1}s_{k}(x)\displaystyle\sum_{l=1}^{2k+2} (-1)^{l-1}(l-1)!\displaystyle\frac{1}{x^{l}}$ is
    $$\sum_{l=0}^{k}S_{l}^{(k)}(-1)^{k+l+1}(k+l+1)!.$$
   Since we have already proved (4.7) and (4.8), we obtain that 
   $$S_{l}^{(k)}=(-1)^{k+l+2}k!\displaystyle\frac{\mathrm{det}A_{k+1,l+1}^{(k)}}{\mathrm{det}A^{(k)}}=\displaystyle\frac{ {}_{k} C _{l}}{l!}\ \ \ \ (l=0,1,\cdots,k).\eqno(4.22)$$
   Therefore we see that
    \begin{eqnarray*}
  \sum_{l=0}^{k}S_{l}^{(k)}(-1)^{k+l+1}(k+l+1)!&=&\sum_{l=0}^{k}\displaystyle\frac{ {}_{k} C _{l}}{l!}(-1)^{k+l+1}(k+l+1)!\\
  &=&\sum_{l=0}^{k} {}_{k} C _{l} (-1)^{k+l+1}(k+l+1)\cdots (l+1)\\
  &=&-(k+1)(k+1)!.
  \end{eqnarray*}
  Then we get that $x^{k+1}r_{k}(x)+x^{k+1}s_{k}(x)\displaystyle\sum_{l=1}^{2k+2} (-1)^{l-1}(l-1)!\displaystyle\frac{1}{x^{l}}$ is congruent to
  $$k!-\frac{(k+1)(k+1)!}{x}$$
  modulo $I^{'}:=\displaystyle\frac{1}{x^{2}}\mathbb{Q}\left[\displaystyle\frac{1}{x}\right]$. Similarly, we have that 
  $$x^{k}r_{k-1}(x)+x^{k}s_{k-1}(x)\displaystyle\sum_{l=1}^{2k+2} (-1)^{l-1}(l-1)!\displaystyle\frac{1}{x^{l}}\equiv (k-1)!-\frac{k k!}{x}\ \ \ \ (\mathrm{mod}I^{'}).$$
 Therefore, substituting them into (4.21), we have that (4.21) is congruent to
 $$-k\left(k!-\frac{(k+1)(k+1)!}{x}\right)-x\left(k!-\frac{(k+1)(k+1)!}{x}\right)+kx\left((k-1)!-\frac{k k!}{x}\right)$$
\begin{eqnarray*}
&\equiv& -kk!-xk!+(k+1)(k+1)!+xk!-k^{2}k!\\
&\equiv& (k+1)(k+1)!-kk!(k+1)\\
&\equiv& (k+1)!(k+1-k)\\
&\equiv& (k+1)!\ \ \ \ (\mathrm{mod}\ I).
\end{eqnarray*}
Therefore, this completes a proof of (4.20). Hence we complete a proof of that
$$F_{2k+2}(x)\sim \frac{1}{k+1}\ \ \ \ (x\to \infty).$$
By induction we complete a proof of Theorem 2.1.\\
\section{A concrete representation of $P_{n}(x)$ and $Q_{n}(x)$}
\ \ \ \ Finally, we give an explicit expression of $P_{n}(x)$ and $Q_{n}(x)$. By (1.4), we obtain the following recurrence relations:
$$ \begin{cases}
P_{-1}(x)=1,\ \ Q_{-1}(x)=0,\ \ P_{0}(x)=0,\ \ Q_{0}(x)=1,\\
P_{n}(x)=m_{n}P_{n-1}(x)+P_{n-2}(x),\ \ Q_{n}(x)=m_{n}Q_{n-1}(x)+Q_{n-2}(x).
\end{cases}
$$
By solving these recurrence relations, we can state the following theorem.
\begin{thm}
For any integer $n$ greater than or equal to 1, it holds that
$$\begin{cases}
P_{2n-1}(x)=\displaystyle\sum_{k=0}^{n-1} \displaystyle\sum_{l=0}^{n-k-1} \displaystyle\frac{{}_{n} C _{l+k+1}}{(l+k)\cdots(l+1)} (-1)^{l}x^{k}, \vspace{5mm}\\
P_{2n}(x)=\displaystyle\sum_{k=0}^{n-1} \displaystyle\sum_{l=0}^{n-k-1}\displaystyle\frac{{}_{n} C _{l+k+1}}{(l+k+1)\cdots(l+1)} (-1)^{l}x^{k}
\end{cases}\ \ \ \ \begin{cases}
Q_{2n-1}(x)=\displaystyle\sum_{k=0}^{n-1}\displaystyle\frac{ {}_{n} C _{k+1}}{k!}x^{k+1}, \vspace{5mm}\\
Q_{2n}(x)=\displaystyle\sum_{k=0}^{n} \displaystyle\frac{ {}_{n} C _{k}}{k!} x^{k}.
\end{cases}$$
\end{thm}
\begin{proof}
We consider the following recurrence relations:
$$ \begin{cases}
Q_{2n+1}(x)=xQ_{2n}(x)+Q_{2n-1}(x),\vspace{3mm}  \\ 
Q_{2n}(x)=\displaystyle\frac{1}{n}Q_{2n-1}(x)+Q_{2n-2}(x),
\end{cases}$$
with initial conditions: $(Q_{0}(x), Q_{1}(x))=(1, x)$. Substituting, we obtain that 
$$Q_{2n+1}(x)=x\left(\displaystyle\frac{1}{n}Q_{2n-1}(x)+Q_{2n-2}(x)\right) +Q_{2n-1}(x)=\left(\displaystyle\frac{x}{n}+1\right)Q_{2n-1}(x)+xQ_{2n-2}(x).$$
Hence we can rewrite as follows:
$$ \begin{cases}
Q_{2n+1}(x)=\left(\displaystyle\frac{x}{n}+1\right)Q_{2n-1}(x)+xQ_{2n-2}(x),\vspace{3mm}  \\ 
Q_{2n}(x)=\displaystyle\frac{1}{n}Q_{2n-1}(x)+Q_{2n-2}(x).
\end{cases}$$
Using a matrix representation, we get that
$$\left(
\begin{array}{cc}
Q_{2n+1}(x) \vspace{9mm}\\
Q_{2n}(x)
\end{array}\right)=
\left(
\begin{array}{cc}
\displaystyle\frac{x}{n}+1 & x \vspace{4mm}\\
\displaystyle\frac{1}{n} & 1
\end{array}
\right)
\left(
\begin{array}{cc}
Q_{2n-1}(x) \vspace{9mm} \\
Q_{2n-2}(x)
\end{array}
\right).
$$
Therefore, we obtain that
\begin{eqnarray*}
  \left(\begin{array}{cc}
Q_{2n+1}(x) \vspace{9mm}\\
Q_{2n}(x)
\end{array}\right)&=&\left(
\begin{array}{cc}
\displaystyle\frac{x}{n}+1 & x \vspace{4mm}\\
\displaystyle\frac{1}{n} & 1
\end{array}
\right)\cdots \left(
    \begin{array}{cc}
     \displaystyle\frac{x}{2}+1 & x   \vspace{4mm} \\
      \displaystyle\frac{1}{2} & 1 
    \end{array}
  \right)\left(
    \begin{array}{cc}
     x+1 &  x   \vspace{9mm} \\
      1 &  1   
    \end{array}
  \right)\left(\begin{array}{cc}
Q_{1}(x) \vspace{9mm}\\
Q_{0}(x)
\end{array}\right)
  \\ \\ \\
  &=& \left(
\begin{array}{cc}
\displaystyle\frac{x}{n}+1 & x \vspace{4mm}\\
\displaystyle\frac{1}{n} & 1
\end{array}
\right)\cdots \left(
    \begin{array}{cc}
     \displaystyle\frac{x}{2}+1 & x   \vspace{4mm} \\
      \displaystyle\frac{1}{2} & 1 
    \end{array}
  \right)\left(
    \begin{array}{cc}
     x+1 &  x   \vspace{9mm} \\
      1 &  1   
    \end{array}
  \right)\left(\begin{array}{cc}
x \vspace{9mm}\\
1
\end{array}\right).
  \end{eqnarray*}
For any integer $n\ge 1$, we put 
$$M_{n}:=\left(
\begin{array}{cc}
\displaystyle\frac{x}{n}+1 & x \vspace{4mm}\\
\displaystyle\frac{1}{n} & 1
\end{array}
\right)\cdots \left(
    \begin{array}{cc}
     \displaystyle\frac{x}{2}+1 & x   \vspace{4mm} \\
      \displaystyle\frac{1}{2} & 1 
    \end{array}
  \right)\left(
    \begin{array}{cc}
     x+1 &  x   \vspace{9mm} \\
      1 &  1   
    \end{array}
  \right)=:\left(
    \begin{array}{cc}
     f_{n}(x) & g_{n}(x)   \vspace{9mm}\\
      h_{n}(x) & i_{n}(x) 
    \end{array}
  \right),$$
   and 
   $$M_{0}:=\left(
    \begin{array}{cc}
     1 \ \ & \ \ 0   \vspace{9mm}\\
      0 \ \ &\ \  1 
    \end{array}
  \right)=:\left(
    \begin{array}{cc}
     f_{0}(x) & g_{0}(x)   \vspace{9mm}\\
      h_{0}(x) & i_{0}(x) 
    \end{array}
  \right).$$
  The same argument for $D_{k}$ in Section4 imply that
  $$M_{n}=\left(
    \begin{array}{cc}
     (x+n)h_{n}(x)-nh_{n-1}(x) &  (x+n)i_{n}(x)-ni_{n-1}(x)  \vspace{9mm}\\
      h_{n}(x) & i_{n}(x) 
    \end{array}
  \right)$$
  and
  $$ \begin{cases}
    h_{n+1}(x)=\displaystyle\frac{x+2n+1}{n+1}h_{n}(x)-\displaystyle\frac{n}{n+1}h_{n-1}(x), \\ \\
   i_{n+1}(x)=\displaystyle\frac{x+2n+1}{n+1}i_{n}(x)-\displaystyle\frac{n}{n+1}i_{n-1}(x),
  \end{cases}$$
  with initial conditions: $(h_{0}(x), h_{1}(x))=(0, 1)$, $(i_{0}(x), i_{1}(x))=(1, 1)$.
  Hence we obtain that
  \begin{eqnarray*}
  M_{n}\left(
    \begin{array}{cc}
    x  \vspace{9mm}\\
     1
    \end{array}
  \right) &=& \left(
    \begin{array}{cc}
    x(x+n)h_{n}(x)-nxh_{n-1}(x) + (x+n)i_{n}(x)-ni_{n-1}(x) \vspace{9mm}\\
      xh_{n}(x)+i_{n}(x) 
    \end{array}
  \right)\\ \\ \\ \\
  &=& \left(
    \begin{array}{cc}
    (x+n)(xh_{n}(x)+i_{n}(x))-n(xh_{n-1}(x)+i_{n-1}(x)) \vspace{9mm}\\
      xh_{n}(x)+i_{n}(x) 
    \end{array}
  \right).
 \end{eqnarray*}
   We set $t_{n}(x)=xh_{n}(x)+i_{n}(x)$. Then $t_{n}(x)$ satisfy the following reccurence relation:
  $$t_{n+1}(x)=\displaystyle\frac{x+2n+1}{n+1}t_{n}(x)-\displaystyle\frac{n}{n+1}t_{n-1}(x),$$
   with initial conditions: $(t_{0}(x),t_{1}(x))=(1,x+1)$. These imply $t_{n}(x)=s_{n}(x)$. Therefore, we have that
   $$\left(\begin{array}{cc}
Q_{2n+1}(x) \vspace{4mm}\\
Q_{2n}(x)
\end{array}\right)=\left(
\begin{array}{cc}
(x+n)s_{n}(x)-ns_{n-1}(x) \vspace{4mm}\\
s_{n}(x)
\end{array}
\right).$$
By (4.22), we obtain that
$$s_{n}(x)=\sum_{k=0}^{n} S_{k}^{(n)}x^{k} =\displaystyle\sum_{k=0}^{n} \displaystyle\frac{ {}_{n} C _{k}}{k!} x^{k}.$$
Therefore, we get that
   \begin{eqnarray*}
  (x+n)s_{n}(x)-ns_{n-1}(x)
  &=&\sum_{k=0}^{n} \frac{ {}_{n} C _{k}}{k!} x^{k+1}+n\left(\sum_{k=0}^{n} \frac{ {}_{n} C _{k}}{k!} x^{k}-\sum_{k=0}^{n-1} \frac{ {}_{n-1} C _{k}}{k!} x^{k}\right)\\ 
   &=& \sum_{k=0}^{n} \frac{ {}_{n} C _{k}}{k!} x^{k+1}+n\left(\displaystyle\frac{1}{n!}x^{n}+\sum_{k=1}^{n-1} \frac{ {}_{n-1} C _{k-1}}{k!} x^{k}\right)\\ 
   &=& \displaystyle\frac{1}{n!}x^{n+1}+\displaystyle\frac{n}{(n-1)!}x^{n}+\sum_{k=0}^{n-2} \frac{ {}_{n} C _{k}}{k!} x^{k+1}+\displaystyle\frac{1}{(n-1)!}x^{n}+\sum_{k=1}^{n-1} \frac{ {}_{n} C _{k}}{(k-1)!} x^{k}\\ 
   &=&\displaystyle\frac{1}{n!}x^{n+1}+\displaystyle\frac{n+1}{(n-1)!}x^{n}+\sum_{k=1}^{n-1} \frac{ {}_{n} C _{k-1}+ {}_{n} C _{k}}{(k-1)!} x^{k}\\ 
  &=&\displaystyle\frac{1}{n!}x^{n+1}+\displaystyle\frac{n+1}{(n-1)!}x^{n}+\sum_{k=1}^{n-1} \frac{{}_{n+1} C _{k}}{(k-1)!} x^{k} \\
   &=&\sum_{k=1}^{n+1} \frac{{}_{n+1} C _{k}}{(k-1)!} x^{k}=\sum_{k=0}^{n} \frac{ {}_{n+1} C _{k+1}}{k!} x^{k+1}.
   \end{eqnarray*}
   Hence, we have that
   $$\left(\begin{array}{cc}
Q_{2n+1}(x) \vspace{4mm}\\
Q_{2n}(x)
\end{array}\right)=\left(
\begin{array}{cc}
\displaystyle\sum_{k=0}^{n} \displaystyle\frac{ {}_{n+1} C _{k+1}}{k!} x^{k+1} \vspace{4mm}\\
\displaystyle\sum_{k=0}^{n} \displaystyle\frac{ {}_{n} C _{k}}{k!} x^{k}
\end{array}
\right).$$
\ \ \ \ We next consider the following recurrence relations:
$$ \begin{cases}
P_{2n+1}(x)=xP_{2n}(x)+P_{2n-1}(x),\vspace{3mm}  \\ 
P_{2n}(x)=\displaystyle\frac{1}{n}P_{2n-1}(x)+P_{2n-2}(x),
\end{cases}$$
with initial conditions: $(P_{0}(x), P_{1}(x))=(0, 1)$. Similarly, we obtain that
$$
  \left(
    \begin{array}{cc}
    P_{2n+1}(x)  \vspace{9mm}\\
     P_{2n}(x)
    \end{array}
  \right) =
  M_{n}\left(
    \begin{array}{cc}
    1  \vspace{9mm}\\
     0
    \end{array}
  \right) = \left(
    \begin{array}{cc}
    (x+n)h_{n}(x)-nh_{n-1}(x)  \vspace{9mm}\\
      h_{n}(x)
    \end{array}
  \right).$$
 Since $h_{0}(x)=-r_{0}(x)$, $h_{1}(x)=-r_{1}(x)$, we see that $h_{n}(x)=-r_{n}(x)$. Then, we have that
 $$ (x+n)h_{n}(x)-nh_{n-1}(x)=n(r_{n-1}(x)-r_{n}(x))-xr_{n}(x)=(n+1)(r_{n}(x)-r_{n+1}(x))$$
 by the recurrence relation $nr_{n-1}(x)=(2n+1+x)r_{n}(x)-(n+1)r_{n+1}(x)$. Substituting (4.22) into (4.4), we obtain that
 $$R^{(n)}_{k}=-\sum_{l=1}^{n-k} \displaystyle\frac{ {}_{n} C _{l+k}}{(l+k)!}(-1)^{l-1}(l-1)!.$$
 Therefore, we have that
 $$-r_{n}(x)=\sum_{k=0}^{n-1} \sum_{l=0}^{n-k-1} \displaystyle\frac{ {}_{n} C _{l+k+1}}{(l+k+1)\cdots (l+1)}(-1)^{l}x^{k}$$
 and this implies
$$ (n+1)(r_{n}(x)-r_{n+1}(x))$$
\begin{eqnarray*}
&=&(n+1)\left(\displaystyle\sum_{k=0}^{n-1} \displaystyle\sum_{l=0}^{n-k-1}\displaystyle\frac{{}_{n} C _{l+k+1}}{(l+k+1)\cdots(l+1)} (-1)^{l}x^{k}-\displaystyle\sum_{k=0}^{n} \displaystyle\sum_{l=0}^{n-k}\displaystyle\frac{{}_{n+1} C _{l+k+1}}{(l+k+1)\cdots(l+1)} (-1)^{l}x^{k}\right)\\
&=&-\displaystyle\frac{1}{n!}x^{n}+(n+1)\displaystyle\sum_{k=0}^{n-1} \left(\displaystyle\frac{(-1)^{n-k-1}}{(n+1)n\cdots (n-k+1)}-\displaystyle\sum_{l=0}^{n-k-1} \displaystyle\frac{{}_{n} C _{l+k}}{(l+k+1)\cdots(l+1)} (-1)^{l-1}\right)x^{k}\\
&=&-\displaystyle\frac{1}{n!}x^{n}-(n+1)\displaystyle\sum_{k=0}^{n-1} \displaystyle\sum_{l=0}^{n-k} \displaystyle\frac{{}_{n} C _{l+k}}{(l+k+1)\cdots(l+1)} (-1)^{l-1}x^{k}\\
&=&-\displaystyle\frac{1}{n!}x^{n}-\displaystyle\sum_{k=0}^{n-1} \displaystyle\sum_{l=0}^{n-k} (l+k+1)\displaystyle\frac{{}_{n+1} C _{l+k+1}}{(l+k+1)\cdots(l+1)} (-1)^{l-1}x^{k}\\
&=&\displaystyle\sum_{k=0}^{n} \displaystyle\sum_{l=0}^{n-k} \displaystyle\frac{{}_{n+1} C _{l+k+1}}{(l+k)\cdots(l+1)} (-1)^{l}x^{k}.
\end{eqnarray*}
Therefore, we see that 
$$(n+1)(r_{n}(x)-r_{n+1}(x))=\displaystyle\sum_{k=0}^{n} \displaystyle\sum_{l=0}^{n-k} \displaystyle\frac{{}_{n+1} C _{l+k+1}}{(l+k)\cdots(l+1)} (-1)^{l}x^{k}.$$
Hence, we obtain that 
$$
  \left(
    \begin{array}{cc}
    P_{2n+1}(x)  \vspace{9mm}\\
     P_{2n}(x)
    \end{array}
  \right) =
  \left(
    \begin{array}{cc}
     \displaystyle\sum_{k=0}^{n} \displaystyle\sum_{l=0}^{n-k} \displaystyle\frac{{}_{n+1} C _{l+k+1}}{(l+k)\cdots(l+1)} (-1)^{l}x^{k}\vspace{9mm}\\
     \displaystyle\sum_{k=0}^{n-1} \displaystyle\sum_{l=0}^{n-k-1} \displaystyle\frac{ {}_{n} C _{l+k+1}}{(l+k)\cdots (l+1)}(-1)^{l}x^{k}
    \end{array}
  \right).$$
  These complete a proof of Theorem 5.1.
\end{proof}
  
\vspace{5mm}
{\bf Address:}  Naoki Murabayashi: Department of Mathematics, Faculty of Engineering Science, Kansai \\
\hspace{10mm}University, 3-3-35, Yamate-cho, Suita-shi, Osaka, 564-8680, Japan.\\
{\bf E-mail:} murabaya@kansai-u.ac.jp \\ \\
{\bf Address:}  Hayato Yoshida: Mathematics, Integrated Science and Engineering Major, Graduate school of Science and\\
\hspace{10mm}Engineering, Kansai University, 3-3-35, Yamate-cho, Suita-shi, Osaka, 564-8680, Japan.\\ 
{\bf E-mail:} k321930@kansai-u.ac.jp
\end{document}